\documentclass[10pt,a4paper]{amsart}
\usepackage[left=3.5cm,right=3.5cm,top=3.5cm,bottom=3.5cm]{geometry}
\usepackage{mathtools}
\usepackage{amsmath,amssymb,amsthm}
\usepackage{hyperref}
\hypersetup{hidelinks}

\DeclareMathOperator{\End}{End}
\DeclareMathOperator{\Hom}{Hom}
\DeclareMathOperator{\Cent}{Cent}
\DeclareMathOperator{\Mat}{Mat}
\DeclareMathOperator{\rk}{rk}
\DeclareMathOperator{\enc}{enc}
\DeclareMathOperator{\wt}{wt}

\theoremstyle{definition}
\newtheorem{definition}{Definition}[section]
\newtheorem{algorithm}[definition]{Algorithm}

\theoremstyle{plain}
\newtheorem{theorem}[definition]{Theorem}
\newtheorem{conjecture}[definition]{Conjecture}
\newtheorem{proposition}[definition]{Proposition}
\newtheorem{corollary}[definition]{Corollary}
\newtheorem{lemma}[definition]{Lemma}

\theoremstyle{remark}
\newtheorem{remark}[definition]{Remark}

\newcommand{\vF}{\mathbb F}
\newcommand{\p}{\mathfrak p}
\renewcommand{\epsilon}{\varepsilon}

\title{A General Construction of Codes from Drinfeld Modules}

\author[A. Giannoni]{Alessandro Giannoni}
\address{Dipartimento di Matematica e Applicazioni ``Renato Caccioppoli'',
Universit\`a degli Studi di Napoli Federico II,
Via Cintia, Monte S. Angelo, I-80126 Napoli, Italy}
\email{alessandro.giannoni@unina.it}

\author[G. Micheli]{Giacomo Micheli}
\address{Department of Mathematics \& Statistics,
University of South Florida,
Tampa, Florida, USA}
\email{gmicheli@usf.edu}

\author[M. Papikian]{Mihran Papikian}
\address{Department of Mathematics, Pennsylvania State University,
University Park, Pennsylvania, United States of America}
\email{papikian@psu.edu}

\subjclass[2020]{Primary 94B05; Secondary 94B35, 11G09, 16S36}

\keywords{Drinfeld modules, rank-metric codes, sum-rank-metric codes,
MRD codes, MSRD codes, skew polynomial rings, decoding}

\begin{document}

\begin{abstract}
We construct additive rank-metric and sum-rank-metric codes from Drinfeld
modules by restricting bounded-degree morphisms to prime-to-characteristic
torsion. For supersingular Drinfeld modules of rank $r$ in characteristic
$\mathfrak p$ of degree $d$, the stabilization formula for morphism spaces
yields rank-metric codes of $\vF_q$-dimension $mrt-c$ and minimum distance
$r-t+1$, where $c=r(r-1)(d-1)/2$. Simultaneous restriction to $\ell$
distinct degree-$m$ torsion modules gives additive sum-rank codes of the same
dimension and minimum distance at least $\ell r-t+1$. Their normalized
Singleton defects tend to zero, while in characteristic $(T)$ the module
$\phi_T=\tau^r$ makes the defect vanish and produces an explicit MSRD family.
We identify this family with a skew Chinese remainder theorem code supported
on central skew polynomials and prove that its poly-skew weight is exactly
$m$ times its sum-rank weight. This gives a specialized Singleton-type bound
and a polynomial-time unique decoder up to the full sum-rank unique-decoding
radius. We also derive a Welch-Berlekamp-type filter equation for the general
supersingular sum-rank construction; it becomes an effective decoder whenever
bases of the relevant morphism spaces and the restriction maps are
computable.
\end{abstract}

\maketitle

\section{Introduction}

Rank-metric codes were introduced independently, in two equivalent
formalisms, by Delsarte and Gabidulin \cite{Del78,Gab85}.  Delsarte studied
subsets of matrix spaces through the association scheme of bilinear forms,
whereas Gabidulin used vectors over an extension field and their rank over the
base field.  Both viewpoints lead to a Singleton-type bound, and codes
attaining it are called maximum rank distance, or MRD, codes.  The classical
Gabidulin construction evaluates $q$-linearized polynomials of bounded
$q$-degree at points that are linearly independent over the base field.  The
fact that the root space of a nonzero linearized polynomial is controlled by
its $q$-degree gives both the minimum distance and, together with an
interpolation argument, the existence of MRD codes throughout the admissible
parameter range.  This root-space principle is also one of the main
ingredients of the constructions in the present paper.  Rank-metric codes
have since become central in random network coding, distributed storage, and
code-based cryptography; see \cite{SKK08,BartzEtAl22} for representative
applications and a broad survey.

For many years, Gabidulin codes and their generalized versions were the only
known linear MRD families available for general parameters.  A major change
came with Sheekey's twisted Gabidulin codes \cite{She16}, which provided
infinite families inequivalent to generalized Gabidulin codes.  This was
followed by additive variants \cite{OO17}, generalized twisted Gabidulin
codes \cite{LTZ18}, constructions arising from maximum scattered linear sets
\cite{CMPZ18}, and the Trombetti-Zhou family \cite{TZ19}.  Skew-polynomial
quotients give another unifying source of MRD codes and semifields, containing
several of the classical and twisted constructions as special cases
\cite{She20}.  These developments revealed particularly fruitful
connections among MRD codes, linearized and skew polynomials, finite
semifields, and scattered subspaces; an overview of this circle of ideas can
be found in \cite{She19}.

The sum-rank metric extends the rank metric from a single matrix to an ordered
collection of matrix blocks by adding their ranks.  It specializes to the
rank metric when there is one block and to the Hamming metric when all blocks
have size $1\times 1$.  The metric appeared naturally in multishot network
coding \cite{NU10} and was subsequently developed as a general coding-theoretic
framework in \cite{MtP18,MtP19,BGLR21,MSK22}.  Its Singleton-type bound
defines the class of maximum sum-rank distance, or MSRD, codes.  The theory of
supports, duality, anticodes, and generalized weights has by now been
developed in substantial generality; see, in particular,
\cite{MtP19,BGLR21,CampsEtAl22}.  MSRD codes play the same extremal role for
the sum-rank metric as MDS and MRD codes do for the Hamming and rank metrics,
respectively.

Linearized Reed-Solomon codes form the basic explicit MSRD family.  They were
introduced through evaluations of skew polynomials and proved to be MSRD over
arbitrary division rings in \cite{MtP18}.  They simultaneously recover
Reed-Solomon codes in the Hamming case and Gabidulin codes in the rank-metric
case.  Their use in reliable and secure multishot network coding, together
with an efficient decoding algorithm, was developed in \cite{MPK19}.
Further constructions have enlarged the available parameter
ranges.  Extended Moore matrices yield several MSRD families with smaller
field sizes \cite{MtP22Moore}; twisted linearized Reed-Solomon codes provide
sum-rank analogues of twisted Gabidulin and Trombetti-Zhou codes
\cite{Neri22}; and extension, product, and modification procedures give
doubly and triply extended families and codes with block sizes not covered by
the original construction \cite{MtP23Extended,MtP24New}.  A systematic
account of the theory, constructions, decoding methods, and applications of
sum-rank codes is given in \cite{MSK22}.

Chinese-remainder constructions for linearized and skew polynomials have also
been developed recently.  Linearized polynomial CRT codes were introduced in
\cite{GGR26} for the rank and sum-rank metrics, with an auxiliary composition
polynomial used to influence their minimum-distance properties.  Skew CRT
codes and the associated poly-skew metric were introduced in \cite{NR26}.  The explicit
characteristic-$(T)$ family considered in the present paper belongs to the
latter formalism for a distinguished collection of central moduli.  In this
case, unlike for general skew CRT supports, the poly-skew metric admits a
direct interpretation as a scaled sum-rank metric on matrix blocks.

We do not determine the sum-rank equivalence class of the explicit
characteristic-$(T)$ family relative to previously known linearized
Reed-Solomon or linearized polynomial CRT codes
\cite{MtP18,MPK19,GGR26}. Accordingly, our novelty claims concern its
Drinfeld-module realization, its identification with a central skew CRT code,
the exact poly-skew/sum-rank metric correspondence, and the resulting
specialized decoder and Singleton-type bound, rather than a new equivalence
class of MSRD codes.

Drinfeld modules offer a different arithmetic source of the same algebraic
features.  A Drinfeld $A$-module is encoded by a homomorphism from
$A=\vF_q[T]$ to a ring of twisted polynomials.  If $f\in A$ is a prime of
degree $m$ away from the $A$-characteristic and $\phi$ has rank $r$, then its
$f$-torsion is a free $A/(f)$-module of rank $r$, and therefore becomes an
$r$-dimensional vector space over $\vF_{q^m}$.  A morphism of Drinfeld modules
commutes with the $A$-action and consequently induces a linear map on these
torsion spaces.  This turns spaces of morphisms into matrix codes in a natural
way.  We refer to \cite{Pap23} for the arithmetic of Drinfeld modules and
their torsion.

The use of Drinfeld modules in rank-metric coding has developed recently.
Bastioni, Darwish, and Micheli used their arithmetic, together with the
Dirichlet theorem for polynomial arithmetic progressions, to construct
infinite families of optimal rank-metric codes with rank-locality
\cite{BDM24}.  Micheli and Papikian then introduced a general
restriction-to-torsion method based on linear spaces of endomorphisms
\cite{MP26Rank}.  Their work focused on semifield codes: it recast a
skew-polynomial construction of Sheekey in Drinfeld-module language and
produced new infinite families from Drinfeld modules over finite fields.
More recently, the same authors determined the dimensions of spaces of
bounded-degree morphisms between supersingular Drinfeld modules.  Put
$\vF_{\mathfrak p}\coloneqq A/\mathfrak p$.  If $\phi$ and $\psi$ are
supersingular of rank $r$ over $\overline{\vF}_{\mathfrak p}$, with
$d=\deg_T(\mathfrak p)$, and
$M_s(\phi,\psi)$ denotes the space of morphisms of $\tau$-degree at most $s$,
their stabilization theorem \cite[Theorem 1.1]{MP26Stable} gives
$$
\dim_{\vF_q}M_s(\phi,\psi)
=
r(s+1)-\frac{r(r-1)(d-1)}{2}
$$
throughout an explicit stable range.  They also conjectured the optimal
threshold \cite[Conjecture 3.10]{MP26Stable} and proved it in rank $2$
\cite[Theorem 1.2]{MP26Stable}.

The purpose of this paper is to develop the restriction-to-torsion viewpoint
beyond the full-rank, semifield setting and to place it in the sum-rank
metric.  Our contributions are as follows.

\begin{itemize}
\item We begin with two Drinfeld modules $\phi$ and $\psi$ of the same rank
$r$, not necessarily supersingular or equal. For a prime $f$ of degree $m$
away from the characteristic, restriction to $f$-torsion defines
$$
\rho_f:\Hom(\phi,\psi)
\longrightarrow
\Mat_{r\times r}(\vF_{q^m}).
$$
Every $\vF_q$-linear subspace $\mathcal M\subseteq\Hom(\phi,\psi)$ therefore
gives an additive rank-metric code. The rank of $\rho_f(u)$ is exactly $r$
minus the $\vF_{q^m}$-dimension of the kernel of $u$ on $\phi[f]$. Thus the
metric problem is reduced to controlling kernels of morphisms on torsion.

\item We apply this construction to the bounded-degree space
$M_{tm-1}(\phi,\psi)$ for supersingular modules in characteristic
$\mathfrak p$. Put
$$
c\coloneqq \frac{r(r-1)(d-1)}{2}.
$$
For $1\leq t<r$ and in the stable range, the resulting additive code has
parameters
$$
\dim_{\vF_q}\mathcal C=mrt-c,
\qquad
d_R(\mathcal C)=r-t+1.
$$
For distinct primes $\mathfrak q_1,\ldots,\mathfrak q_\ell$ of degree $m$,
simultaneous restriction of bounded-degree endomorphisms gives a code in
$\Mat_{r\times r}(\vF_{q^m})^\ell$.  The Chinese remainder decomposition of
the torsion turns the sum of the kernel dimensions in the blocks into a
global root count, and for $1\leq t\leq \ell r$ we obtain
$$
\dim_{\vF_q}\mathcal C=mrt-c,
\qquad
d_{SR}(\mathcal C)\geq \ell r-t+1.
$$
In both constructions the normalized additive Singleton defect tends to zero.

\item We specialize to characteristic $(T)$ and the explicit supersingular
module $\phi_T=\tau^r$. Here $c=0$, so the construction is MSRD. With
$R=\vF_{q^r}\{\tau\}$ and $Q_i=\mathfrak q_i(\tau^r)$, the polynomials $Q_i$
are central and define a skew CRT realization
$$
R/(Q_i)
\cong
\End_{\vF_{q^m}}\bigl(\phi[\mathfrak q_i]\bigr)
\cong
\Mat_{r\times r}(\vF_{q^m}).
$$
For this support we prove the exact metric identity
$$
\wt_{\mathrm{ps}}=m\wt_{SR}
$$
and the corresponding specialized Singleton-type bound.

\item For the general supersingular sum-rank construction, we derive a
Welch-Berlekamp-type filter equation.  Under two explicit dimension and
degree conditions, the equation has a nonzero solution and every nonzero
solution recovers the transmitted endomorphism.  This gives an
existence-and-uniqueness result in arbitrary finite characteristic.  It
becomes an effective decoder whenever bases of the relevant bounded-degree
endomorphism spaces and the corresponding restriction matrices can be
computed. For the characteristic-$(T)$ family, the skew CRT key equation gives
a polynomial-time
unique decoder for all sum-rank errors of weight at most
$$
\left\lfloor\frac{\ell r-t}{2}\right\rfloor.
$$
\end{itemize}

The sum-rank construction is naturally adapted to a multishot channel with a
global error budget.  We recall the usual correction criterion and compare a
joint MSRD strategy with protecting every shot by an independent MRD code.
For errors of total sum-rank at most $e$, the joint strategy has normalized
rate $1-2e/(\ell r)$, whereas the separated strategy has normalized rate
$1-2e/r$.  This elementary comparison explains why simultaneous evaluation
on several torsion modules is the appropriate coding model.

The paper is organized as follows.  Section~2 recalls rank-metric and
sum-rank-metric codes, discusses the multishot channel, and collects the
required background on Drinfeld modules.  Section~3 develops the general
restriction-to-torsion construction, applies the stabilization theorem to
bounded-degree morphism spaces, derives the asymptotically MRD and MSRD
families, and concludes with the exact MSRD construction in characteristic
$(T)$.  Section~4 derives a filter equation for the general supersingular
sum-rank construction, identifies the characteristic-$(T)$ family with a
skew CRT code, proves the equivalence of the poly-skew and sum-rank metrics
in this setting, and gives a polynomial-time bounded-distance decoder for the
explicit family.

\section{Background}

\subsection{Rank-Metric and Sum-Rank Codes}

We begin by recalling the rank-metric and sum-rank notions used throughout the paper. The rank metric goes back to Delsarte and Gabidulin, while the sum-rank metric was introduced in the context of multishot network coding and later developed as a general metric framework; see \cite{Del78,Gab85,NU10,MtP18,MtP19}. Throughout this subsection, ranks of matrices are taken over the ambient field.

\begin{definition}
Let $K$ be a field and let $m,n$ be positive integers. The \emph{rank distance} on $\Mat_{m\times n}(K)$ is
$$
d_R(X,Y)\coloneqq \rk(X-Y)
$$
for all $X,Y\in \Mat_{m\times n}(K)$.
\end{definition}

\begin{definition}
Let $K$ be a field. A \emph{rank-metric code} is a subset
$$
\mathcal C\subseteq \Mat_{m\times n}(K)
$$
equipped with the rank distance. If $\mathcal C$ is a nonzero $K$-linear subspace, its minimum rank distance is
$$
d_R(\mathcal C)
\coloneqq
\min\{\rk(X):0\neq X\in\mathcal C\}.
$$
\end{definition}

\begin{definition}
Let $K$ be a field and let $\mathcal C\subseteq \Mat_{m\times n}(K)$ be a nonzero $K$-linear rank-metric code. The \emph{rank-metric Singleton bound} is
$$
\dim_K(\mathcal C)\leq \max\{m,n\}\bigl(\min\{m,n\}-d_R(\mathcal C)+1\bigr).
$$
A rank-metric code attaining this bound is called a \emph{maximum rank distance code}, or an \emph{MRD code}.
\end{definition}

\begin{definition}
Let $K=\vF_{q^m}$ and let $\mathcal C\subseteq \Mat_{a\times b}(K)$. We say that $\mathcal C$ is an \emph{additive rank-metric code} if it is an $\vF_q$-linear subspace of $\Mat_{a\times b}(K)$. If $\mathcal C$ is nonzero, its minimum rank distance is
$$
d_R(\mathcal C)
\coloneqq
\min\{\rk(X):0\neq X\in\mathcal C\}.
$$
For an additive rank-metric code, the Singleton bound is
$$
\dim_{\vF_q}(\mathcal C)
\leq
m\max\{a,b\}\bigl(\min\{a,b\}-d_R(\mathcal C)+1\bigr).
$$
If equality holds, then $\mathcal C$ is called an \emph{additive MRD code}.
\end{definition}

\begin{definition}
Let $K$ be a field and let
$$
\mathcal V
\coloneqq
\Mat_{m_1\times n_1}(K)\times\cdots\times \Mat_{m_\ell\times n_\ell}(K).
$$
For
$$
X=(X_1,\ldots,X_\ell)\in\mathcal V,
$$
the \emph{sum-rank weight} of $X$ is
$$
\wt_{SR}(X)\coloneqq \sum_{i=1}^{\ell}\rk(X_i).
$$
The associated \emph{sum-rank distance} is
$$
d_{SR}(X,Y)\coloneqq \wt_{SR}(X-Y)
=
\sum_{i=1}^{\ell}\rk(X_i-Y_i).
$$
\end{definition}

\begin{definition}
Let $K$ be a field and let
$$
\mathcal V
=
\Mat_{m_1\times n_1}(K)\times\cdots\times \Mat_{m_\ell\times n_\ell}(K).
$$
A \emph{sum-rank code} is a subset
$$
\mathcal C\subseteq \mathcal V
$$
equipped with the sum-rank distance. If $\mathcal C$ is a nonzero $K$-linear subspace, its minimum sum-rank distance is
$$
d_{SR}(\mathcal C)
\coloneqq
\min\{\wt_{SR}(X):0\neq X\in\mathcal C\}.
$$
\end{definition}

\begin{definition}
Let $K$ be a field and let
$$
\mathcal V=\Mat_{r\times r}(K)^\ell.
$$
For a nonzero $K$-linear sum-rank code $\mathcal C\subseteq\mathcal V$, the \emph{sum-rank Singleton bound} is
$$
\dim_K(\mathcal C)\leq r\bigl(\ell r-d_{SR}(\mathcal C)+1\bigr).
$$
A sum-rank code attaining this bound is called a \emph{maximum sum-rank distance code}, or an \emph{MSRD code}.
\end{definition}

\begin{definition}
Let $K=\vF_{q^m}$ and let $\mathcal C\subseteq \Mat_{r\times r}(K)^\ell$. We say that $\mathcal C$ is an \emph{additive sum-rank code} if it is an $\vF_q$-linear subspace of $\Mat_{r\times r}(K)^\ell$. If $\mathcal C$ is nonzero, its minimum sum-rank distance is
$$
d_{SR}(\mathcal C)
\coloneqq
\min\{\wt_{SR}(X):0\neq X\in\mathcal C\}.
$$
For an additive sum-rank code, the Singleton bound is
$$
\dim_{\vF_q}(\mathcal C)
\leq
mr\bigl(\ell r-d_{SR}(\mathcal C)+1\bigr).
$$
If equality holds, then $\mathcal C$ is called an \emph{additive MSRD code}.
\end{definition}

For the rank-metric and sum-rank Singleton bounds, including their additive
forms, see \cite{Del78,MtP18,CampsEtAl22}.

In the constructions below, the resulting codes are generally additive: they are $\vF_q$-linear subspaces of matrix spaces over $\vF_{q^m}$, but need not be $\vF_{q^m}$-linear. We shall therefore use the additive versions of the rank-metric and sum-rank Singleton bounds.

\subsection{The Multishot Sum-Rank Channel}

The sum-rank metric is particularly well suited to multishot models, as in multishot network coding; see \cite{NU10,MPK19}.

Let $K$ be a field. We consider a multishot transmission in which a codeword is an $\ell$-tuple of matrices
$$
X=(X_1,\ldots,X_\ell)\in \Mat_{r\times r}(K)^\ell.
$$
The received word is
$$
Y=(Y_1,\ldots,Y_\ell),
\qquad
Y_i=X_i+E_i,
$$
where $E_i\in \Mat_{r\times r}(K)$ is the error occurring in the $i$-th shot. The \emph{bounded sum-rank error channel} with parameter $e$ is the channel in which the admissible error tuples satisfy
$$
\sum_{i=1}^{\ell}\rk(E_i)\leq e.
$$

The usual nearest-neighbor argument gives the following correction criterion.

\begin{proposition}
\label{prop:sum-rank-channel-correction}
Let $\mathcal C\subseteq \Mat_{r\times r}(K)^\ell$ be a sum-rank code, and define
$$
d_{SR}(\mathcal C)
\coloneqq
\min\{d_{SR}(X,X'):X,X'\in\mathcal C,\ X\neq X'\}.
$$
Then $\mathcal C$ corrects all errors in the bounded sum-rank error channel with parameter $e$ whenever
$$
2e<d_{SR}(\mathcal C).
$$
\end{proposition}

\begin{proof}
Let $X,X'\in\mathcal C$ be distinct codewords, and suppose that a received word $Y$ is within sum-rank distance at most $e$ from both $X$ and $X'$. Then
$$
d_{SR}(X,X')
\leq
d_{SR}(X,Y)+d_{SR}(Y,X')
\leq 2e,
$$
contradicting $2e<d_{SR}(\mathcal C)$.
\end{proof}

We next compare joint coding across the shots with a separated strategy that protects each shot independently.

\begin{proposition}
\label{prop:joint-sum-rank-rate-advantage}
Assume that $\ell>1$ and let $e$ be a positive integer such that $2e<r$. Consider the bounded sum-rank error channel on $\Mat_{r\times r}(K)^\ell$ with parameter $e$. If one uses a joint MSRD code with minimum sum-rank distance $2e+1$, then its normalized rate is
$$
R_{SR}=1-\frac{2e}{\ell r}.
$$
On the other hand, if one uses $\ell$ independent one-shot MRD codes in $\Mat_{r\times r}(K)$, each with minimum rank distance $2e+1$, in order to correct every error tuple satisfying
$$
\sum_{i=1}^{\ell}\rk(E_i)\leq e,
$$
then the normalized rate is
$$
R_{\mathrm{sep}}=1-\frac{2e}{r}.
$$
In particular,
$$
R_{SR}-R_{\mathrm{sep}}
=
\frac{2e(\ell-1)}{\ell r}
>
0.
$$
\end{proposition}

\begin{proof}
For the joint code, Proposition~\ref{prop:sum-rank-channel-correction} shows that correction of all errors of sum-rank weight at most $e$ is guaranteed by minimum sum-rank distance $2e+1$. An MSRD code in $\Mat_{r\times r}(K)^\ell$ with this distance has dimension
$$
r\bigl(\ell r-(2e+1)+1\bigr)
=
r(\ell r-2e).
$$
Since the ambient space has dimension $\ell r^2$ over $K$, its normalized rate is
$$
R_{SR}
=
\frac{r(\ell r-2e)}{\ell r^2}
=
1-\frac{2e}{\ell r}.
$$

For the separated strategy, the global constraint allows all the error to be concentrated in a single shot: for some $j$, one may have $\rk(E_j)=e$ and $E_i=0$ for all $i\neq j$. Hence each one-shot component code must correct rank errors of rank at most $e$, and therefore must have rank distance at least $2e+1$. Taking each component code to be MRD with minimum rank distance $2e+1$, its dimension is
$$
r\bigl(r-(2e+1)+1\bigr)
=
r(r-2e).
$$
Using $\ell$ such independent component codes gives total dimension $\ell r(r-2e)$ inside an ambient space of dimension $\ell r^2$. Thus
$$
R_{\mathrm{sep}}
=
\frac{\ell r(r-2e)}{\ell r^2}
=
1-\frac{2e}{r}.
$$
The strict inequality follows by subtracting the two rates.
\end{proof}

\begin{remark}
The comparison in Proposition~\ref{prop:joint-sum-rank-rate-advantage} is with a separated coding strategy on a genuinely multishot channel. If the physical channel instead allows one to send a single large matrix and the error is measured by ordinary rank, then the natural comparison is different. Put $N=r\ell$. A one-shot MRD code in $\Mat_{N\times N}(K)$ with minimum rank distance $2e+1$ corrects all rank errors of rank at most $e$ and has dimension
$$
N\bigl(N-(2e+1)+1\bigr)=N(N-2e).
$$
Its normalized rate is therefore
$$
R_{\mathrm{MRD}}
=
\frac{N(N-2e)}{N^2}
=
1-\frac{2e}{N}
=
1-\frac{2e}{r\ell}.
$$
This equals the normalized rate $R_{SR}$ of a Singleton-optimal sum-rank code in
$$
\Mat_{r\times r}(K)^\ell
$$
with minimum sum-rank distance $2e+1$. Thus the advantage of the sum-rank model is not that it improves on classical MRD codes for a single large rank-metric transmission, but rather that it matches the structure of a channel naturally split into $\ell$ shots with a global sum-rank error budget.
\end{remark}

MSRD codes, and codes whose dimension is close to the MSRD bound, are natural for this channel because they maximize, or nearly maximize, the normalized rate for a prescribed sum-rank correction capability.

\subsection{Drinfeld Modules}

We now recall the basic language of Drinfeld modules. We follow the standard notation for $A=\vF_q[T]$; for background, see \cite{Pap23}.

\begin{definition}
Let $q$ be a prime power, let $\vF_q$ be the finite field with $q$ elements, and put
$$
A\coloneqq \vF_q[T].
$$
An $A$-\emph{field} is a field $k\supseteq \vF_q$ equipped with an $\vF_q$-algebra homomorphism
$$
\gamma:A\longrightarrow k.
$$
The ideal $\ker(\gamma)$ is called the $A$-\emph{characteristic} of $k$. If $\ker(\gamma)=\{0\}$, then $k$ is said to have \emph{generic} $A$-characteristic. If $\ker(\gamma)\neq\{0\}$, then $k$ is said to have \emph{finite} $A$-characteristic, and we write
$$
\p\coloneqq \ker(\gamma).
$$
\end{definition}

Throughout the rest of the paper, when a nonzero ideal of $A$ has a unique monic generator, we also use the same symbol for that generator whenever it appears in a polynomial expression. Thus, for a finite $A$-characteristic $\p$, expressions such as $\deg_T(\p)$, $\phi_\p$, and $f\neq\p$ refer to the monic generator of $\p$.

The noncommutative polynomial ring below encodes additive polynomials and will be used to write Drinfeld module actions.

\begin{definition}
Let $k$ be a commutative $\vF_q$-algebra and let $\tau$ be an indeterminate. The ring of \emph{twisted polynomials} $k\{\tau\}$ consists of the polynomials
$$
\sum_{i=0}^n a_i\tau^i,
\qquad a_i\in k,
$$
with the usual addition and multiplication determined by
$$
(a\tau^i)(b\tau^j)=ab^{q^i}\tau^{i+j}
$$
for all $a,b\in k$ and all $i,j\geq 0$.
\end{definition}

Equivalently, twisted polynomials may be viewed as $q$-linearized polynomials under the correspondence $\tau^i\mapsto Z^{q^i}$.

\begin{definition}
Let $k$ be a commutative $\vF_q$-algebra. A polynomial
$$
L(Z)=\sum_{i=0}^r a_iZ^{q^i}\in k[Z]
$$
is called a $q$-\emph{linearized} polynomial. If $a_r\neq 0$, then $L$ has $q$-\emph{degree} $r$. The ring of $q$-linearized polynomials over $k$, with addition and composition, is denoted by $k\langle Z\rangle$.
\end{definition}

We can now define Drinfeld modules in terms of this twisted polynomial ring.

\begin{definition}
Let $k$ be an $A$-field with structure morphism $\gamma:A\rightarrow k$. A \emph{Drinfeld module} over $k$ is an $\vF_q$-algebra homomorphism
$$
\phi:A\longrightarrow k\{\tau\},
\qquad
a\longmapsto \phi_a,
$$
such that, for every $a\in A$,
$$
\phi_a=\gamma(a)+\sum_{i=1}^{n(a)}h_i(a)\tau^i
$$
with $h_i(a)\in k$, and such that $\phi$ is not equal to $\gamma$ viewed as a map $A\rightarrow k\subseteq k\{\tau\}$. If there exists an integer $r\geq 1$ such that
$$
\deg_\tau(\phi_a)=r\deg_T(a)
$$
for every nonconstant $a\in A$, then $\phi$ is said to have \emph{rank} $r$.
\end{definition}

Because $A$ is a polynomial ring in one variable, a Drinfeld module is determined by the image of $T$.

\begin{definition}
Let $\phi$ be a Drinfeld module over an $A$-field $k$. Since $A=\vF_q[T]$, the homomorphism $\phi$ is determined by $\phi_T$. Thus, if $\phi$ has rank $r$, then
$$
\phi_T=\gamma(T)+g_1\tau+\cdots+g_r\tau^r,
$$
where $g_i\in k$ and $g_r\neq 0$. Conversely, such a choice of $\phi_T$ determines a Drinfeld module of rank $r$.
\end{definition}

The homomorphism $\phi$ also gives an $A$-module structure on the additive group.

\begin{definition}
Let $\phi$ be a Drinfeld module over an $A$-field $k$. The additive group of $k$ becomes an $A$-module, denoted by ${}^{\phi}k$, through the action
$$
a\circ \beta \coloneqq \phi_a(\beta)
$$
for every $a\in A$ and every $\beta\in k$.
\end{definition}

The next definition fixes the convention on base change for morphisms.

\begin{definition}
Let $\phi,\psi$ be Drinfeld modules over an $A$-field $k$. A \emph{morphism} $u:\phi\rightarrow\psi$ over $k$ is an element $u\in k\{\tau\}$ such that
$$
u\phi_a=\psi_a u
$$
for every $a\in A$. The set of morphisms from $\phi$ to $\psi$ over $k$ is denoted by $\Hom_k(\phi,\psi)$. Let $\phi_{\overline{k}}$ and $\psi_{\overline{k}}$ denote the base changes of $\phi$ and $\psi$ to an algebraic closure $\overline{k}$. We write
$$
\Hom(\phi,\psi)\coloneqq \Hom_{\overline{k}}(\phi_{\overline{k}},\psi_{\overline{k}})
$$
for the group of geometric morphisms. The \emph{endomorphism ring} of $\phi$ over $k$ is
$$
\End_k(\phi)\coloneqq \Hom_k(\phi,\phi)
=
\Cent_{k\{\tau\}}(\phi(A)),
$$
and we put
$$
\End(\phi)\coloneqq \Hom(\phi,\phi).
$$
\end{definition}

We shall use the corresponding torsion modules as evaluation spaces for morphisms.

\begin{definition}
Let $\phi$ be a Drinfeld module over an $A$-field $k$, and let $a\in A$ be nonzero. An element
$$
\alpha\in \overline{k}
$$
is called an $a$-\emph{torsion point} of $\phi$ if
$$
\phi_a(\alpha)=0.
$$
The set of all $a$-torsion points of $\phi$ is denoted by
$$
\phi[a]\coloneqq \{\alpha\in\overline{k}:\phi_a(\alpha)=0\}.
$$
\end{definition}

\begin{definition}
Let $\phi$ be a Drinfeld module of rank $r$ over an $A$-field $k$ of finite $A$-characteristic $\p$, and let $d\coloneqq \deg_T(\p)$. The \emph{height} of $\phi$ is the integer $h$, with $1\leq h\leq r$, such that the first nonzero term of $\phi_\p$ has $\tau$-degree $hd$; equivalently,
$$
\phi_\p=\sum_{i=hd}^{rd} a_i\tau^i,
\qquad
a_{hd}\neq 0.
$$
If $k$ has generic $A$-characteristic, then $\phi$ is said to have height $0$. This notion is unchanged after extending the base field.
\end{definition}

The height gives the standard notion of supersingularity in finite $A$-characteristic.

\begin{definition}
Let $\phi$ be a Drinfeld module of rank $r$ over an $A$-field $k$ of finite $A$-characteristic $\p$. The Drinfeld module $\phi$ is called \emph{supersingular} if its height is equal to its rank, namely if
$$
h(\phi)=r.
$$
\end{definition}

Over finite fields there is also a distinguished endomorphism coming from the field Frobenius.

\begin{definition}
Let $\phi$ be a Drinfeld module of rank $r$ over a finite field $k=\vF_{q^n}$ containing $\vF_q$. The \emph{Frobenius endomorphism} of $\phi$ over $k$ is
$$
\pi\coloneqq \tau^n\in \End_k(\phi).
$$
Indeed, $\tau^n$ commutes with all coefficients in $k$, and hence with $\phi(A)$.
\end{definition}

\section{The Constructions}

We now explain how restriction of morphisms to prime-to-characteristic torsion produces additive rank-metric and sum-rank codes.

\subsection{From Morphisms to Rank-Metric Codes}

Let $\phi$ and $\psi$ be Drinfeld modules of the same rank $r$ over a field
$k\supseteq \vF_q$.  Let $f\in A$ be a monic irreducible polynomial of degree
$m$, prime to the $A$-characteristic of $k$.  Since $f$ is prime to the
$A$-characteristic, both $\phi_f$ and $\psi_f$ are separable, and
$$
    \phi[f]\cong (A/(f))^r\cong \vF_{q^m}^r,
    \qquad
    \psi[f]\cong (A/(f))^r\cong \vF_{q^m}^r.
$$
See \cite[Corollary 3.5.3]{Pap23}.

The compatibility of morphisms with the $A$-action implies that they preserve torsion in the expected way.

\begin{proposition}
Let $u\in \Hom(\phi,\psi)$. Then $u$ maps $\phi[f]$ into $\psi[f]$.
\end{proposition}

\begin{proof}
If $y\in \phi[f]$, then $\phi_f(y)=0$. Since $u\in\Hom(\phi,\psi)$, we have $u\phi_f=\psi_f u$. Hence
$$
\psi_f(u(y))=u(\phi_f(y))=u(0)=0.
$$
Therefore $u(y)\in\psi[f]$.
\end{proof}

Thus restriction to $\phi[f]$ gives a natural homomorphism
$$
\rho_f:\Hom(\phi,\psi)\longrightarrow \Hom_{A/(f)}(\phi[f],\psi[f]).
$$
After choosing $A/(f)$-bases of $\phi[f]$ and $\psi[f]$, and using the identification $A/(f)\cong \vF_{q^m}$, we obtain a matrix representation, still denoted by
$$
\rho_f:\Hom(\phi,\psi)\longrightarrow \Mat_{r\times r}(\vF_{q^m}).
$$
The resulting matrices depend on the chosen bases, but a change of bases only acts by left and right multiplication by invertible matrices. Hence the rank metric parameters are independent of these choices.

This gives the basic rank-metric construction.

\begin{definition}
Let $\mathcal M\subseteq \Hom(\phi,\psi)$ be an $\vF_q$-linear subspace. The additive rank-metric code obtained from $\mathcal M$ and the pair $(\phi[f],\psi[f])$ is
$$
\mathcal C_{\mathcal M,f}\coloneqq \rho_f(\mathcal M)
\subseteq \Mat_{r\times r}(\vF_{q^m}).
$$
Equivalently, the associated restriction map is
$$
\enc:\mathcal M\longrightarrow \Mat_{r\times r}(\vF_{q^m}),
\qquad
u\longmapsto \left(u:\phi[f]\rightarrow\psi[f]\right).
$$
When $\rho_f|_{\mathcal M}$ is injective, this map may be viewed as an encoding map with message space $\mathcal M$; otherwise the corresponding message space is naturally $\mathcal M/\ker(\rho_f|_{\mathcal M})$.
\end{definition}

For $u\in\mathcal M$, the matrix $\rho_f(u)$ represents the $\vF_{q^m}$-linear map
$$
u:\phi[f]\cong \vF_{q^m}^r\longrightarrow\psi[f]\cong \vF_{q^m}^r.
$$
Therefore
$$
\rk_{\vF_{q^m}}(\rho_f(u))
=
r-\dim_{\vF_{q^m}}\ker\left(u|_{\phi[f]}\right).
$$
Consequently, if there exists an integer $t$ such that
$$
\dim_{\vF_{q^m}}\ker\left(u|_{\phi[f]}\right)\leq t
$$
for every nonzero $u\in\mathcal M$, then
$$
d_R(\mathcal C_{\mathcal M,f})\geq r-t.
$$

Thus the construction reduces the problem of constructing additive rank-metric codes to that of finding $\vF_q$-linear subspaces $\mathcal M\subseteq \Hom(\phi,\psi)$ whose nonzero elements have small kernels on a chosen torsion module. The construction does not require $\phi$ or $\psi$ to be supersingular. Supersingularity becomes useful as a source of large and computable spaces of morphisms, such as the bounded-degree spaces discussed in the next section, but the restriction-to-torsion construction itself applies to arbitrary Drinfeld modules.

\subsection{Bounded-Degree Isogeny Spaces}

To obtain explicit families, we filter geometric morphisms by their $\tau$-degree.

\begin{definition}
Let $\phi$ and $\psi$ be Drinfeld modules over an $A$-field $k$. Following Micheli and Papikian, we work with the geometric morphism space
$$
\Hom(\phi,\psi)\coloneqq \Hom_{\overline{k}}(\phi,\psi).
$$
For an integer $s\geq 0$, set
$$
M_s(\phi,\psi)
\coloneqq
\{0\}\cup\{u\in \Hom(\phi,\psi):u\neq 0,\ \deg_\tau(u)\leq s\}.
$$
This is an $\vF_q$-linear subspace of $\Hom(\phi,\psi)$. When $\phi=\psi$, we write
$$
M_s(\phi)\coloneqq M_s(\phi,\phi)
=
\{0\}\cup\{u\in \End(\phi):u\neq 0,\ \deg_\tau(u)\leq s\}.
$$
\end{definition}

The key input is the stabilization theorem of Micheli and Papikian
\cite[Theorem 1.1]{MP26Stable}, stated here in the form needed below.

\begin{theorem}
Let $\vF_{\p}\coloneqq A/\p$.  Let $\phi$ and $\psi$ be supersingular
Drinfeld modules of rank $r\geq 2$ over $\overline{\vF}_{\p}$, and let
$$
d\coloneqq \deg_T(\p).
$$
Then
$$
\dim_{\vF_q} M_s(\phi,\psi)
=
r(s+1)-\frac{r(r-1)(d-1)}{2}
$$
for every integer $s$ satisfying
$$
s\geq \frac{r^2(r-1)(d-1)}{2}.
$$
In particular,
$$
\dim_{\vF_q} M_s(\phi)
=
r(s+1)-\frac{r(r-1)(d-1)}{2}
$$
in the same range.
\end{theorem}

The stable range above is obtained from
\cite[Proposition 3.19]{MP26Stable}.  Micheli and Papikian further conjecture
that the same formula holds in a sharper range
\cite[Conjecture 3.10]{MP26Stable}.

\begin{conjecture}
Let $\vF_{\p}\coloneqq A/\p$.  Let $\phi$ and $\psi$ be supersingular
Drinfeld modules of rank $r\geq 2$ over $\overline{\vF}_{\p}$, and let
$d=\deg_T(\p)$.  Then
$$
\dim_{\vF_q} M_s(\phi,\psi)
=
r(s+1)-\frac{r(r-1)(d-1)}{2}
$$
for every integer $s\geq 0$ satisfying
$$
s\geq (r-1)(d-1)-1.
$$
\end{conjecture}

The conjectured range holds for $r=2$ by
\cite[Theorem 1.2]{MP26Stable}.

The following root-counting lemma isolates the common mechanism behind the
rank-metric and sum-rank constructions.

\begin{lemma}
\label{lem:torsion-root-count}
Let $\phi$ and $\psi$ be Drinfeld modules of the same rank over an $A$-field,
and, for $i=1,\ldots,\ell$, let $\mathfrak q_i$ be distinct monic
irreducible polynomials of degree $m$, prime to the $A$-characteristic. For
every nonzero $u\in\Hom(\phi,\psi)$,
$$
\sum_{i=1}^{\ell}
\dim_{\vF_{q^m}}
\ker\left(u|_{\phi[\mathfrak q_i]}\right)
\leq
\left\lfloor\frac{\deg_\tau(u)}{m}\right\rfloor.
$$
\end{lemma}

\begin{proof}
Put
$$
k_i\coloneqq
\dim_{\vF_{q^m}}
\ker\left(u|_{\phi[\mathfrak q_i]}\right).
$$
The Chinese remainder decomposition of the torsion identifies the direct sum
of these kernels with an $\vF_q$-subspace of the root space of $u$ of
dimension $m\sum_i k_i$. A nonzero twisted polynomial of $\tau$-degree
$\deg_\tau(u)$ has at most $q^{\deg_\tau(u)}$ roots in an algebraic closure.
Consequently,
$$
m\sum_{i=1}^{\ell}k_i\leq\deg_\tau(u),
$$
and the result follows because the left hand side divided by $m$ is an
integer.
\end{proof}

\subsection{An Asymptotically MRD Construction}

Let $\vF_{\mathfrak p}\coloneqq A/\mathfrak p$.  Let $\phi$ and $\psi$ be
supersingular Drinfeld modules of rank $r\geq 2$ over
$\overline{\vF}_{\mathfrak p}$, and let
$$
d\coloneqq\deg_T(\mathfrak p).
$$
Let $f\in A$ be a monic irreducible polynomial of degree $m$, with
$$
f\neq \mathfrak p.
$$
Since $f$ is prime to the $A$-characteristic, the torsion modules $\phi[f]$ and $\psi[f]$ are separable. Fix an integer $t$ such that
$$
1\leq t<r,
$$
and set
$$
s\coloneqq tm-1.
$$
Assume that $s$ lies in the stabilization range of the previous theorem. Equivalently, with
$$
c\coloneqq\frac{r(r-1)(d-1)}{2},
$$
we assume that
$$
tm-1\geq rc.
$$
Consider the $\vF_q$-linear space
$$
M_s\coloneqq M_s(\phi,\psi)
=
\{0\}\cup\{u\in \Hom(\phi,\psi):u\neq 0,\ \deg_\tau(u)\leq s\}.
$$
Restricting the elements of $M_s$ to the $f$-torsion gives an additive rank-metric code
$$
\mathcal C_{s,f}\coloneqq \rho_f(M_s)\subseteq \Mat_{r\times r}(\vF_{q^m}).
$$

The next proposition records the parameters obtained from this bounded-degree space.

\begin{proposition}
For the code $\mathcal C_{s,f}$, one has
$$
\dim_{\vF_q}\mathcal C_{s,f}=mrt-c
$$
and
$$
d_R(\mathcal C_{s,f})=r-t+1.
$$
In particular, $\mathcal C_{s,f}$ is exactly $c$ below the additive Singleton bound for its minimum rank distance.
\end{proposition}

\begin{proof}
For every nonzero $u\in M_s$, Lemma~\ref{lem:torsion-root-count}, applied
with $\ell=1$ and $\mathfrak q_1=f$, gives
$$
\dim_{\vF_{q^m}}\ker\left(u|_{\phi[f]}\right)
\leq
\left\lfloor\frac{tm-1}{m}\right\rfloor
=t-1.
$$
Therefore
$$
\rk_{\vF_{q^m}}\rho_f(u)
=
r-\dim_{\vF_{q^m}}\ker\left(u|_{\phi[f]}\right)
\geq r-t+1.
$$
Hence
$$
d_R(\mathcal C_{s,f})\geq r-t+1.
$$

We now compute the dimension. By the stabilization theorem,
$$
\dim_{\vF_q}M_s
=
r(s+1)-c
=
rtm-c.
$$
Moreover, the restriction map $\rho_f$ is injective on $M_s$. Indeed, if $0\neq u\in M_s$ were in the kernel of $\rho_f$, then $u|_{\phi[f]}=0$, so
$$
\dim_{\vF_{q^m}}\ker\left(u|_{\phi[f]}\right)=r.
$$
This contradicts the preceding bound, since $t<r$. Therefore $\rho_f$ is injective on $M_s$, and consequently
$$
\dim_{\vF_q}\mathcal C_{s,f}
=
\dim_{\vF_q}M_s
=
mrt-c.
$$

It remains to show that the lower bound on the minimum distance is sharp. Suppose, for a contradiction, that
$$
d_R(\mathcal C_{s,f})\geq r-t+2.
$$
Then the additive Singleton bound for rank-metric codes in $\Mat_{r\times r}(\vF_{q^m})$ gives
$$
\dim_{\vF_q}\mathcal C_{s,f}
\leq
mr\bigl(r-(r-t+2)+1\bigr)
=
mr(t-1).
$$
On the other hand, we have proved that
$$
\dim_{\vF_q}\mathcal C_{s,f}=mrt-c.
$$
The stabilization-range hypothesis gives
$$
tm-1\geq rc,
$$
and hence
$$
c\leq \frac{tm-1}{r}<\frac{tm}{r}.
$$
Since $t<r$, we have $c<m$, and in particular $c<mr$. Therefore
$$
mrt-c>mrt-mr=mr(t-1),
$$
contradicting the Singleton bound above. Hence
$$
d_R(\mathcal C_{s,f})=r-t+1.
$$

For this value of the minimum distance, the additive Singleton bound is
$$
\dim_{\vF_q}\mathcal C\leq mr\bigl(r-(r-t+1)+1\bigr)=mrt.
$$
Since
$$
\dim_{\vF_q}\mathcal C_{s,f}=mrt-c,
$$
the code $\mathcal C_{s,f}$ is exactly $c$ below the additive Singleton bound for its actual minimum rank distance.
\end{proof}

Thus, for fixed $r$, $t$, and $d$, the normalized Singleton defect is
$$
\frac{c}{mrt-c},
$$
which tends to zero as $m\to\infty$. In this normalization, the family is
asymptotically MRD, or equivalently asymptotically Singleton-optimal.

\subsection{An Asymptotically MSRD Construction}

Let $\vF_{\p}\coloneqq A/\p$.  Let $\phi$ be a supersingular Drinfeld module
of rank $r\geq 2$ over $\overline{\vF}_{\p}$, and let
$$
d\coloneqq \deg_T(\p).
$$
Assume that there exist at least $\ell$ monic irreducible polynomials of degree $m$ in $A$ different from $\p$, and choose distinct such polynomials
$$
\mathfrak q_1,\ldots,\mathfrak q_\ell\in A.
$$
Since each level $\mathfrak q_i$ is prime to the $A$-characteristic, the torsion modules $\phi[\mathfrak q_i]$ are separable. Set
$$
P\coloneqq \mathfrak q_1\cdots \mathfrak q_\ell.
$$
Since $P$ is prime to the $A$-characteristic, the prime-to-characteristic torsion structure gives
$$
\phi[P]\cong (A/P)^r.
$$
By the coprimality of the $\mathfrak q_i$'s, this identifies the $P$-torsion as
$$
\phi[P]\cong \bigoplus_{i=1}^{\ell}\phi[\mathfrak q_i],
\qquad
\phi[\mathfrak q_i]\cong (A/\mathfrak q_i)^r\cong \vF_{q^m}^r.
$$
For each $i$, we choose an isomorphism $A/\mathfrak q_i\cong \vF_{q^m}$ and an $A/\mathfrak q_i$-basis of $\phi[\mathfrak q_i]$. These choices allow us to represent the restrictions to the $\mathfrak q_i$-torsion as matrices in $\Mat_{r\times r}(\vF_{q^m})$. Changing these choices only applies rank-preserving transformations to the blocks, and hence does not affect the sum-rank parameters.

Fix an integer $t$ such that
$$
1\leq t\leq \ell r,
$$
and set
$$
s\coloneqq tm-1.
$$
Assume that $s$ lies in the stabilization range of the previous theorem, namely
$$
tm-1\geq \frac{r^2(r-1)(d-1)}{2}.
$$
Consider the $\vF_q$-linear space
$$
\mathcal M_s\coloneqq M_s(\phi)
=
\{0\}\cup\{u\in \End(\phi):u\neq 0,\ \deg_\tau(u)\leq s\}.
$$
Restricting the elements of $\mathcal M_s$ to the torsion modules $\phi[\mathfrak q_i]$ gives an additive sum-rank code
$$
\mathcal C_{s,\mathfrak q_1,\ldots,\mathfrak q_\ell}
\coloneqq
\left\{
\left(
u|_{\phi[\mathfrak q_1]},\ldots,u|_{\phi[\mathfrak q_\ell]}
\right)
:
u\in\mathcal M_s
\right\}
\subseteq
\Mat_{r\times r}(\vF_{q^m})^\ell.
$$

\begin{theorem}
The simultaneous restriction map is injective on $\mathcal M_s$, and the
resulting additive code satisfies
$$
\dim_{\vF_q}\mathcal C_{s,\mathfrak q_1,\ldots,\mathfrak q_\ell}
=mrt-c,
\qquad
d_{SR}(\mathcal C_{s,\mathfrak q_1,\ldots,\mathfrak q_\ell})
\geq\ell r-t+1,
$$
where $c=r(r-1)(d-1)/2$. If $d_{\mathrm{act}}$ denotes the actual minimum
sum-rank distance and
$$
\Delta_{\mathrm{act}}
\coloneqq
mr(\ell r-d_{\mathrm{act}}+1)-(mrt-c),
$$
then
$$
0\leq\Delta_{\mathrm{act}}\leq c.
$$
Consequently, for fixed $r$ and $d$,
$$
0\leq
\frac{\Delta_{\mathrm{act}}}{mrt-c}
\leq
\frac{c}{mrt-c}
\longrightarrow0
$$
whenever $mt\to\infty$. Thus the family is asymptotically MSRD in this
normalization, or equivalently asymptotically Singleton-optimal.
\end{theorem}

\begin{proof}
For $0\neq u\in\mathcal M_s$, put
$$
k_i(u)\coloneqq
\dim_{\vF_{q^m}}\ker\left(u|_{\phi[\mathfrak q_i]}\right).
$$
Lemma~\ref{lem:torsion-root-count} gives
$$
\sum_{i=1}^{\ell}k_i(u)
\leq
\left\lfloor\frac{tm-1}{m}\right\rfloor
=t-1.
$$
Therefore every nonzero codeword has sum-rank weight at least
$$
\sum_{i=1}^{\ell}(r-k_i(u))
\geq\ell r-t+1.
$$
Since $t\leq\ell r$, the same inequality excludes a nonzero element of
$\mathcal M_s$ restricting to zero. The restriction map is therefore
injective, and the stabilization theorem gives
$$
\dim_{\vF_q}\mathcal C_{s,\mathfrak q_1,\ldots,\mathfrak q_\ell}
=
\dim_{\vF_q}\mathcal M_s
=mrt-c.
$$
The additive Singleton bound gives $\Delta_{\mathrm{act}}\geq0$. On the
other hand, $d_{\mathrm{act}}\geq\ell r-t+1$ implies
$$
mr(\ell r-d_{\mathrm{act}}+1)\leq mrt,
$$
and hence $\Delta_{\mathrm{act}}\leq c$. The normalized inequalities and
their limit follow immediately.
\end{proof}

\subsection{An Explicit MSRD Construction in Characteristic
\texorpdfstring{$(T)$}{(T)}}

We now specialize the previous construction to a case where the Singleton
defect vanishes. Let $A=\vF_q[T]$ and let $k=\vF_{q^r}$, endowed with the
$A$-field structure for which $T$ acts as zero. Consider the Drinfeld module
$$
    \phi_T=\pi=\tau^r.
$$
Then $\phi$ has rank $r$, its $A$-characteristic is the ideal $(T)$, and
$$
    d=\deg_T(T)=1.
$$
In particular, $\phi$ is supersingular.
Moreover,
$$
    \End_k(\phi)=\End(\phi)=\vF_{q^r}\{\tau\},
$$
since a twisted polynomial $\sum_j a_j\tau^j$ commutes with $\tau^r$ if and
only if $a_j^{q^r}=a_j$ for every $j$.

For an integer $m\geq 1$, let
$$
I_q(m)
\coloneqq
\frac{1}{m}\sum_{j\mid m}\mu(j)q^{m/j}
$$
be the number of monic irreducible polynomials of degree $m$ in
$\vF_q[T]$, where $\mu$ denotes the Möbius function: $\mu(1)=1$, $\mu(j)=0$ if $j$ is
divisible by the square of a prime, and $\mu(j)=(-1)^k$ if $j$ is the
product of $k$ distinct primes. Hence the admissible
number of blocks satisfies
$$
1\leq\ell\leq
\begin{cases}
q-1, & m=1,\\
I_q(m), & m>1.
\end{cases}
$$
In particular, $I_q(m)\sim q^m/m$ as $m\to\infty$.

Fix any such $\ell$, and choose distinct monic irreducible polynomials
$$
\mathfrak q_1,\ldots,\mathfrak q_\ell\in A
$$
of degree $m$, none equal to $T$. Set
$$
    P\coloneqq \mathfrak q_1\cdots \mathfrak q_\ell.
$$
Since $P$ is prime to the $A$-characteristic, the prime-to-characteristic
torsion decomposes as
$$
    \phi[P]\cong \bigoplus_{i=1}^{\ell}\phi[\mathfrak q_i],
    \qquad
    \phi[\mathfrak q_i]\cong (A/\mathfrak q_i)^r\cong \vF_{q^m}^r.
$$
These summands will be the blocks for the sum-rank metric.

Fix an integer $t$ with
$$
    1\leq t\leq \ell r
$$
and set
$$
    s\coloneqq tm-1.
$$
Let
$$
    \mathcal M_s\coloneqq M_s(\phi)
    =
    \{0\}\cup
    \{u\in \End_k(\phi):u\neq 0,\ \deg_\tau(u)\leq s\}.
$$
Restricting endomorphisms to the torsion blocks gives the code
$$
    \mathcal C^{\mathrm{MSRD}}_{s,\mathfrak q_1,\ldots,\mathfrak q_\ell}
    \coloneqq
    \left\{
    \left(
    u|_{\phi[\mathfrak q_1]},\ldots,u|_{\phi[\mathfrak q_\ell]}
    \right)
    :
    u\in\mathcal M_s
    \right\}
    \subseteq
    \Mat_{r\times r}(\vF_{q^m})^\ell.
$$

\begin{theorem}
\label{thm:explicit-msrd-characteristic-t}
The code
$\mathcal C^{\mathrm{MSRD}}_{s,\mathfrak q_1,\ldots,\mathfrak q_\ell}$ is an
additive MSRD code. More precisely,
$$
    d_{SR}\left(
    \mathcal C^{\mathrm{MSRD}}_{s,\mathfrak q_1,\ldots,\mathfrak q_\ell}
    \right)
    =
    \ell r-t+1
$$
and
$$
    \dim_{\vF_q}
    \mathcal C^{\mathrm{MSRD}}_{s,\mathfrak q_1,\ldots,\mathfrak q_\ell}
    =
    mrt.
$$
\end{theorem}

\begin{proof}
Let $0\neq u\in\mathcal M_s$, and write
$$
    k_i(u)\coloneqq
    \dim_{\vF_{q^m}}\ker\left(u|_{\phi[\mathfrak q_i]}\right).
$$
Lemma~\ref{lem:torsion-root-count} gives
$$
\sum_{i=1}^{\ell}k_i(u)
\leq
\left\lfloor\frac{tm-1}{m}\right\rfloor
=t-1.
$$
Therefore every nonzero codeword has sum-rank weight at least
$$
    \sum_{i=1}^{\ell}(r-k_i(u))
    =
    \ell r-\sum_{i=1}^{\ell}k_i(u)
    \geq
    \ell r-t+1.
$$

Since $t\leq \ell r$, the restriction map is injective on $\mathcal M_s$.
Indeed, a nonzero element restricting to the zero codeword would have
$\sum_i k_i(u)=\ell r$, contradicting the previous bound. Thus
$$
    \dim_{\vF_q}
    \mathcal C^{\mathrm{MSRD}}_{s,\mathfrak q_1,\ldots,\mathfrak q_\ell}
    =
    \dim_{\vF_q}\mathcal M_s.
$$
Since
$$
\End_k(\phi)=\vF_{q^r}\{\tau\},
$$
the space $\mathcal M_s$ is an $\vF_{q^r}$-vector space with basis
$$
1,\tau,\ldots,\tau^s.
$$
Therefore
$$
    \dim_{\vF_q}\mathcal M_s
    =
    r(s+1)
    =
    rtm.
$$

Let $d_{SR}$ denote the actual minimum sum-rank distance of the code.  The
additive sum-rank Singleton bound gives
$$
    mrt
    =
    \dim_{\vF_q}\mathcal C^{\mathrm{MSRD}}_{s,\mathfrak q_1,\ldots,\mathfrak q_\ell}
    \leq
    mr(\ell r-d_{SR}+1).
$$
Therefore $d_{SR}\leq\ell r-t+1$.  Together with the lower bound already
proved, this gives $d_{SR}=\ell r-t+1$.  Thus the code attains the additive
sum-rank Singleton bound and is an additive MSRD code.
\end{proof}

\section{Filter Equations and Skew CRT Decoding}

We first give a Welch-Berlekamp-type filter equation and the associated
decoder for the general supersingular sum-rank construction. We then specialize to
$\phi_T=\tau^r$, where the code has a native skew CRT realization, the
poly-skew and sum-rank metrics agree up to the factor $m$, and the skew CRT
key equation gives an effective decoder.

\subsection{A Filter Equation for the General Sum-Rank Construction}

Let $\vF_{\mathfrak p}\coloneqq A/\mathfrak p$. Let $\phi$ be a
supersingular Drinfeld module of rank $r\geq 2$ over
$\overline{\vF_{\mathfrak p}}$, put
$d\coloneqq\deg_T(\mathfrak p)$, and retain the notation of the asymptotically
MSRD construction in Section~3. Thus
$\mathfrak q_1,\ldots,\mathfrak q_\ell$ are distinct primes of degree $m$,
prime to $\mathfrak p$, and $s=tm-1$. Put
$$
P\coloneqq\prod_{i=1}^{\ell}\mathfrak q_i.
$$
For $i=1,\ldots,\ell$,
$$
\rho_i:\End(\phi)
\longrightarrow
\End_{A/(\mathfrak q_i)}\bigl(\phi[\mathfrak q_i]\bigr)
\cong
\Mat_{r\times r}(\vF_{q^m})
$$
is restriction to $\mathfrak q_i$-torsion. Write a received word as
$$
Y_i=\rho_i(u)+E_i,
\qquad
u\in M_s(\phi).
$$

The unknown $v$ below plays the role of an error filter. Ideally it
annihilates the image of every block error $E_i$; in that case
$\rho_i(v)Y_i=\rho_i(vu)$. Since neither the errors nor such a $v$ are
known, we introduce a second unknown $N$, representing the product $vu$, and
solve simultaneously for $(v,N)$. The resulting equations are linear over
$\vF_q$ after bases of the bounded-degree morphism spaces have been fixed.

\begin{proposition}
\label{prop:general-filter-equation}
Let $\epsilon\in\mathbb Z_{\geq 0}$. Suppose that
$\lambda\in\mathbb Z_{\geq 0}$ satisfies
$$
\dim_{\vF_q}M_\lambda(\phi)>mr\epsilon,
\qquad
\lambda+s<m(\ell r-\epsilon).
$$
If $\sum_i\rk_{\vF_{q^m}}(E_i)\leq\epsilon$, then the linear equations
$$
\rho_i(v)Y_i=\rho_i(N),
\qquad i=1,\ldots,\ell,
$$
have a nonzero solution
$(v,N)\in M_\lambda(\phi)\times M_{\lambda+s}(\phi)$, and every nonzero
solution satisfies $N=vu$. In particular, $v\neq0$ and $u$ is the right
quotient in this identity.
\end{proposition}

\begin{proof}
Put $U_i\coloneqq\operatorname{Im}(E_i)$ and
$\nu\coloneqq\sum_i\dim_{\vF_{q^m}}(U_i)\leq\epsilon$. The map
$$
M_\lambda(\phi)
\longrightarrow
\bigoplus_{i=1}^{\ell}
\Hom_{\vF_{q^m}}\bigl(U_i,\phi[\mathfrak q_i]\bigr),
\qquad
v\longmapsto\bigl(\rho_i(v)|_{U_i}\bigr)_i
$$
has a codomain of $\vF_q$-dimension $mr\nu\leq mr\epsilon$. Hence it has a
nonzero kernel element $v$, and $(v,vu)$ is a nonzero solution.

For an arbitrary nonzero solution set $Z\coloneqq N-vu$. If
$x\in\ker(E_i)$, the $i$-th filter equation gives $Z(x)=0$. Thus $Z$
vanishes on $\bigoplus_i\ker(E_i)\subseteq\phi[P]$, whose
$\vF_q$-dimension is at least
$$
m\sum_{i=1}^{\ell}
\left(r-\rk_{\vF_{q^m}}(E_i)\right)
\geq m(\ell r-\epsilon).
$$
Since $\deg_\tau(Z)\leq\lambda+s<m(\ell r-\epsilon)$, the root-space bound
forces $Z=0$. Hence $N=vu$, and $v\neq0$ because the solution is nonzero.
\end{proof}

The first condition on $\lambda$ ensures that a nonzero filter exists; the
second forces $Z=N-vu$ to vanish by the root-space bound. Thus $\lambda$
balances existence and uniqueness.

In the stabilization range this budget can be chosen explicitly. Set
$$
c\coloneqq\frac{r(r-1)(d-1)}{2},
\qquad
a\coloneqq\left\lfloor\frac{c}{r}\right\rfloor
=
\left\lfloor\frac{(r-1)(d-1)}{2}\right\rfloor.
$$
Taking $\lambda=m\epsilon+a$, the stabilization formula gives
$$
\dim_{\vF_q}M_\lambda(\phi)-mr\epsilon=r(a+1)-c>0,
$$
and the second hypothesis of Proposition~\ref{prop:general-filter-equation}
becomes
$$
m(2\epsilon+t)+a\leq m\ell r.
$$
In characteristic $(T)$, this reduces to $2\epsilon+t\leq\ell r$. The
choice $\lambda=m\epsilon+a$ is valid only when it belongs to the
stabilization range, namely when
$$
m\epsilon+a\geq \frac{r^2(r-1)(d-1)}{2}.
$$
For effectiveness, fix $\vF_q$-bases of $M_\lambda(\phi)$ and
$M_{\lambda+s}(\phi)$, matrices for the $\rho_i$, and algorithms for
finite-field and Ore arithmetic and membership testing. The filter equations
are then $\vF_q$-linear.

\begin{algorithm}
\label{alg:general-filter-decoder}
Fix a decoding radius $\epsilon\in\mathbb Z_{\geq 0}$ and a filter-degree
bound $\lambda\in\mathbb Z_{\geq 0}$ such that
$\dim_{\vF_q}M_\lambda(\phi)>mr\epsilon$ and
$\lambda+s<m(\ell r-\epsilon)$.

\emph{Input:} a received word $\boldsymbol Y=(Y_1,\ldots,Y_\ell)$ with
$Y_i\in\Mat_{r\times r}(\vF_{q^m})$.  \emph{Output:} a candidate
$\widehat u\in M_s(\phi)$ at sum-rank distance at most $\epsilon$ from
$\boldsymbol Y$, or failure.

\begin{enumerate}
\item \emph{Filter step.} Solve simultaneously
$$
\rho_i(v)Y_i=\rho_i(N),
\qquad i=1,\ldots,\ell,
$$
for the filter $v\in M_\lambda(\phi)$ and filtered numerator
$N\in M_{\lambda+s}(\phi)$. In the fixed bases this is a homogeneous linear
system over $\vF_q$. If its only solution is $(0,0)$, declare failure;
otherwise choose any nonzero solution and declare failure if $v=0$.

\item \emph{Message-recovery step.} Divide $N$ with $v$ on the left in the
ambient Ore-polynomial ring, obtaining
$$
N=v\widehat u+\Delta,
\qquad
\deg_\tau(\Delta)<\deg_\tau(v).
$$
Here $\widehat u$ is the quotient on the right. Declare failure if
$\Delta\neq0$ or $\widehat u\notin M_s(\phi)$.

\item \emph{Verification step.} Compute the residual sum-rank weight. Declare
failure if
$$
\sum_{i=1}^{\ell}
\rk_{\vF_{q^m}}\bigl(Y_i-\rho_i(\widehat u)\bigr)>\epsilon;
$$
otherwise return $\widehat u$.
\end{enumerate}
\end{algorithm}
\begin{corollary}
\label{cor:general-filter-radius}
Let $\lambda\in\mathbb Z_{\geq 0}$ satisfy $\lambda+s<m\ell r$, and set
$$
\epsilon_\lambda
\coloneqq
\min\left\{
\left\lfloor
\frac{\dim_{\vF_q}M_\lambda(\phi)-1}{mr}
\right\rfloor,
\left\lfloor
\frac{m(\ell r-t)-\lambda}{m}
\right\rfloor
\right\}.
$$
Then Algorithm~\ref{alg:general-filter-decoder}, with decoding radius
$\epsilon_\lambda$, corrects every error tuple of sum-rank weight at most
$\epsilon_\lambda$.

In characteristic $(T)$, taking
$$
\lambda
=
m\left\lfloor\frac{\ell r-t}{2}\right\rfloor
$$
gives
$$
\epsilon_\lambda
=
\left\lfloor\frac{\ell r-t}{2}\right\rfloor.
$$
Thus, in this case, the general filter decoder reaches the full
unique-decoding radius.
\end{corollary}

\begin{proof}
The definition of $\epsilon_\lambda$ gives
$$
\dim_{\vF_q}M_\lambda(\phi)>mr\epsilon_\lambda
$$
and, since $s=tm-1$,
$$
\lambda+s<m(\ell r-\epsilon_\lambda).
$$
The claim therefore follows from
Proposition~\ref{prop:general-filter-equation}. In characteristic $(T)$ one
has
$$
\dim_{\vF_q}M_\lambda(\phi)=r(\lambda+1),
$$
and the stated choice of $\lambda$ gives the final equality.
\end{proof}
 With the effective data fixed, the algorithm is polynomial:
it uses $\vF_q$-linear algebra followed by Ore-polynomial arithmetic. This is
a conditional effective statement, not a uniform algorithm for arbitrary
supersingular Drinfeld modules given without such data.

\begin{remark}
Let $(\phi,\psi)$ be as in the bounded-degree rank-metric construction and
let $\epsilon\in\mathbb Z_{\geq 0}$. Suppose
$$
Y=\rho_f^{\phi,\psi}(u)+E,
\qquad
u\in M_s(\phi,\psi),
\qquad
\rk_{\vF_{q^m}}(E)\leq\epsilon.
$$
If
$$
\dim_{\vF_q}M_\lambda(\psi)>mr\epsilon,
\qquad
\lambda+s<m(r-\epsilon),
$$
one solves
$$
\rho_f^\psi(v)Y=\rho_f^{\phi,\psi}(N),
\qquad
(v,N)\in M_\lambda(\psi)\times M_{\lambda+s}(\phi,\psi).
$$
The same proof gives $N=vu$ for every nonzero solution, and the same
division-and-verification procedure decodes the bounded-degree rank-metric
codes whenever the required bases and restriction matrices are effective.
\end{remark}

\subsection{Skew CRT Realization}

We use throughout the notation of our characteristic-$(T)$ construction and
formulate the skew CRT construction of \cite{NR26} directly in this setting.
Put
$$
R=\vF_{q^r}\{\tau\},
\qquad
\tau\alpha=\alpha^q\tau
\quad(\alpha\in\vF_{q^r}),
\qquad
Q_i=\mathfrak q_i(\tau^r),
$$
$$
\Gamma=(Q_1,\ldots,Q_\ell),
\qquad
\Pi=P(\tau^r)=\prod_{i=1}^{\ell}Q_i,
\qquad
\mathsf N\coloneqq\deg_\tau(\Pi)=m\ell r.
$$
The ring $R$ is Euclidean on both sides. If $F,G\in R$, with $G$ monic and
nonzero, right Euclidean division is the unique decomposition
$$
F=BG+H,
\qquad
\deg_\tau(H)<\deg_\tau(G).
$$
We write $H=\operatorname{rem}_r(F,G)$. Thus the left-module quotient
$R/RG$ records right residues modulo $G$. For nonzero $F,G\in R$, we use
monic greatest common right divisors and least common left multiples, with
$$
\operatorname{gcrd}(F,G)=B_1F+B_2G,
\qquad
\operatorname{lclm}(F,G)=C_1F=C_2G,
$$
$$
\deg_\tau\operatorname{lclm}(F,G)
=
\deg_\tau(F)+\deg_\tau(G)
-\deg_\tau\operatorname{gcrd}(F,G).
$$
Multiplication in $R$ corresponds to composition of the associated
$q$-linearized polynomials: the product $FG$ acts as $F\circ G$. Congruences
are defined by right remainders, so the relevant quotients are initially left
module quotients.
As shown in the next proposition, the polynomials $Q_i$ and $\Pi$ are
central; hence the ideals they generate are two-sided and the quotients below
are rings. By left division of $N$ by $v$ we mean
$N=v\widehat u+\Delta$ with
$\deg_\tau(\Delta)<\deg_\tau(v)$: the divisor $v$ is on the left and the
quotient $\widehat u$ is on the right.

\begin{proposition}
The elements $Q_1,\ldots,Q_\ell$ are central and pairwise coprime, and
$$
\Pi=\operatorname{lclm}(Q_1,\ldots,Q_\ell),
\qquad
\deg_\tau(\Pi)=\sum_i\deg_\tau(Q_i).
$$
Thus $\Gamma$ is $\Pi$-independent in the terminology of
\cite[Definition~4]{NR26}. Moreover,
restriction to $\mathfrak q_i$-torsion induces an isomorphism
$$
\Theta_i:R/(Q_i)
\xrightarrow{\sim}
\End_{A/(\mathfrak q_i)}\bigl(\phi[\mathfrak q_i]\bigr)
\cong
\Mat_{r\times r}(\vF_{q^m}).
$$
Set
$$
\mathcal R_\Gamma\coloneqq\prod_{i=1}^{\ell}R/(Q_i)
$$
and, by the skew CRT theorem \cite[Theorem~1]{NR26}, denote the resulting
isomorphism by
$$
\operatorname{CRT}_\Gamma:R/(\Pi)\xrightarrow{\sim}\mathcal R_\Gamma,
\qquad
u\longmapsto\bigl(\operatorname{rem}_r(u,Q_i)\bigr)_i.
$$
We suppress residue-class notation in the arguments of
$\operatorname{CRT}_\Gamma$ and identify each class in $R/(Q_i)$ with its
unique right-remainder representative of degree less than $mr$.
Following \cite[Definition~5]{NR26}, for $1\leq K\leq\mathsf N$ let
$$
\operatorname{SCRT}_{\Gamma,K}
\coloneqq
\left\{\operatorname{CRT}_\Gamma(u):\deg_\tau(u)<K\right\}.
$$
This is an $\vF_{q^r}$-linear code of dimension $K$. Define the blockwise
matrix map
$$
\Theta:\mathcal R_\Gamma
\xrightarrow{\sim}
\Mat_{r\times r}(\vF_{q^m})^\ell,
\qquad
(y_i)_i\longmapsto\bigl(\Theta_i(y_i)\bigr)_i.
$$
Then
$$
\Theta\bigl(\operatorname{SCRT}_{\Gamma,mt}\bigr)
=
\mathcal C^{\mathrm{MSRD}}_{s,\mathfrak q_1,\ldots,\mathfrak q_\ell}.
$$
\end{proposition}

\begin{proof}
Since $\alpha\mapsto\alpha^q$ has order $r$ on $\vF_{q^r}$,
$\tau^r$ is central. Hence $Q_i\in\vF_q[\tau^r]$; the coprimality of the
$\mathfrak q_i$ remains valid after substituting $T=\tau^r$, and their
product is their least common left multiple.

Every element of $R=\End_k(\phi)$ commutes with the $A$-action, so
restriction factors through $R/(Q_i)$. If $u$ vanishes on
$\phi[\mathfrak q_i]$, write by right division
$$
u=hQ_i+w,
\qquad
\deg_\tau(w)<mr.
$$
Since $\mathfrak q_i$ is prime to the $A$-characteristic $(T)$, the torsion
module $\phi[\mathfrak q_i]$ has $q^{mr}$ elements, all of which are roots
of $w$. Since $\deg_\tau(w)<mr$, the root-space bound gives $w=0$. The
induced map is injective and is surjective because both sides have
$\vF_q$-dimension $mr^2$. The product CRT isomorphism and the identity
$\mathcal M_s=M_s(\phi)=\{u\in R:\deg_\tau(u)<mt\}$ give the final identity.
\end{proof}

Here $\vF_{q^r}$-linearity refers to left multiplication in the
skew-polynomial residue model. Under $\Theta$, multiplication by
$\alpha\in\vF_{q^r}$ becomes left composition by the matrices
$\Theta_i(\alpha)$; it need not coincide with entrywise scalar multiplication
over $\vF_{q^m}$.
Thus the corresponding matrix code is asserted to be $\vF_q$-additive, while
the $\vF_{q^r}$-vector-space structure is intrinsic to its skew CRT model.

The encoding process can therefore be read from left to right in the diagram
$$
\{u\in R:\deg_\tau(u)<K\}
\xrightarrow{\ \operatorname{CRT}_\Gamma\ }
\mathcal R_\Gamma
\xrightarrow{\ \Theta\ }
\Mat_{r\times r}(\vF_{q^m})^\ell.
$$
The first map records the skew remainders of $u$ modulo the $Q_i$; the second
map identifies each remainder with the matrix of its action on the
corresponding torsion module. Thus a message is represented by a skew
polynomial of degree less than $K$, while $\mathsf N-K$ is the unused degree
range that supplies redundancy.

\subsection{The Metric Correspondence and the Singleton Bound}

For $\boldsymbol y=(y_i)_i\in\mathcal R_\Gamma$,
let $Y$ be its CRT lift of degree less than $\mathsf N$. Its poly-skew weight
is
$$
\wt_{\mathrm{ps}}(\boldsymbol y)
\coloneqq
\mathsf N-\deg_\tau\operatorname{gcrd}(\Pi,Y),
$$
as in \cite[Definition~6]{NR26}.

The quantity $\deg_\tau\operatorname{gcrd}(\Pi,Y)$ measures how much of the
torsion space $\ker(\Pi)$ is also annihilated by $Y$. Accordingly, the
poly-skew weight measures the complementary part on which $Y$ acts
nontrivially. The next theorem makes this interpretation precise by showing
that it is exactly $m$ times the sum of the matrix ranks of the residue
blocks.

\begin{theorem}
\label{thm:poly-skew-sum-rank-isometry}
For every $\boldsymbol y\in\mathcal R_\Gamma$,
$$
\wt_{\mathrm{ps}}(\boldsymbol y)
=
m\wt_{SR}\bigl(\Theta(\boldsymbol y)\bigr).
$$
Thus $\Theta$ is an isometry after rescaling the poly-skew metric by $1/m$.
\end{theorem}

\begin{proof}
Put $D=\operatorname{gcrd}(\Pi,Y)$. Since $D$ divides both $\Pi$ and $Y$ on
the right, one has
$$
\ker(D)\subseteq\ker(\Pi)\cap\ker(Y).
$$
Conversely, a Bezout identity $D=B_1\Pi+B_2Y$ gives the reverse inclusion.
Moreover,
$$
\Pi=P(\tau^r)=\phi_P,
\qquad
\ker(\Pi)=\phi[P].
$$
Consequently,
$$
\ker(D)=\ker(\Pi)\cap\ker(Y)
=
\ker\left(Y|_{\phi[P]}\right).
$$
The constant coefficient of $\Pi$ is nonzero, since
$\mathfrak q_i\neq T$ for every $i$. If $\Pi=HD$, then the constant term of
$\Pi$ is the product of the constant terms of $H$ and $D$; hence the constant
coefficient of $D$ is also nonzero. Thus $D$ is separable and
$$
\deg_\tau(D)
=
\dim_{\vF_q}\ker\left(Y|_{\phi[P]}\right).
$$
Since $\phi[P]=\bigoplus_i\phi[\mathfrak q_i]$ and the action on the $i$-th
summand is represented by $\Theta_i(y_i)$,
$$
\deg_\tau(D)
=
m\sum_{i=1}^{\ell}
\left(r-\rk_{\vF_{q^m}}\bigl(\Theta_i(y_i)\bigr)\right)
=
\mathsf N-m\wt_{SR}\bigl(\Theta(\boldsymbol y)\bigr).
$$
The claim follows from the definition of $\wt_{\mathrm{ps}}$.
\end{proof}

\begin{corollary}
Let $\mathcal D\subseteq\mathcal R_\Gamma$ be a nonzero $\vF_q$-linear code
of dimension $\kappa$ and minimum poly-skew distance $d_{\mathrm{ps}}$.
Then
$$
\kappa
\leq
mr\left(\ell r-\frac{d_{\mathrm{ps}}}{m}+1\right).
$$
If $\mathcal D$ is $\vF_{q^r}$-linear of dimension $K$, then
$$
d_{\mathrm{ps}}
\leq
m\left(\ell r-\left\lceil\frac{K}{m}\right\rceil+1\right).
$$
For the explicit Drinfeld-module code, $K=mt$ and
$$
d_{\mathrm{ps}}\left(\operatorname{SCRT}_{\Gamma,mt}\right)
=
m(\ell r-t+1)
=
\mathsf N-K+m.
$$
\end{corollary}

\begin{proof}
Apply the additive sum-rank Singleton bound to $\Theta(\mathcal D)$ and use
Theorem~\ref{thm:poly-skew-sum-rank-isometry}. If $\mathcal D$ is
$\vF_{q^r}$-linear, then $\kappa=rK$. Moreover, the isometry shows that
$d_{\mathrm{ps}}/m$ is an integer. Dividing the additive bound by $r$ and
rounding therefore gives
$$
\frac{d_{\mathrm{ps}}}{m}
\leq
\ell r-\left\lceil\frac{K}{m}\right\rceil+1.
$$
For the explicit code,
Theorem~\ref{thm:explicit-msrd-characteristic-t} gives
$d_{SR}=\ell r-t+1$.
\end{proof}

For $K=mt$, the central support explains why the exact distance is
$\mathsf N-K+m$, rather than merely the general skew CRT designed-distance
lower bound $\mathsf N-K+1$ from \cite[Proposition~5]{NR26}.

\subsection{A Skew CRT Decoder for the Explicit
Characteristic-\texorpdfstring{$(T)$}{(T)} Codes}

The key equation is the skew CRT analogue of Welch-Berlekamp reconstruction
\cite{Gab85,Loi04,Loi06}. It specializes the reconstruction mechanism of
\cite[Propositions~6 and~7]{NR26} to the present central support. Once the
maps $\Theta_i$ are fixed, a received matrix tuple is converted to skew
residues, lifted by CRT, and decoded entirely in $R$.

Under the identifications $\Theta_i$, the general filter equations become
particularly concrete. For a received matrix tuple
$\boldsymbol Y=(Y_1,\ldots,Y_\ell)$, set
$\boldsymbol y=(y_1,\ldots,y_\ell)
\coloneqq\Theta^{-1}(\boldsymbol Y)$ and let $Y$ be the unique CRT lift of
$\boldsymbol y$ satisfying $\deg_\tau(Y)<\mathsf N$. Then
$$
\rho_i(v)Y_i=\rho_i(N)
$$
is equivalent to
$$
\operatorname{rem}_r(vy_i,Q_i)
=
\operatorname{rem}_r(N,Q_i).
$$
Since $Y\equiv y_i\pmod{Q_i}$ for every $i$, all the block equations are
equivalent, by the skew CRT theorem, to the single congruence
$$
vY\equiv N\pmod{\Pi}.
$$
Thus the decoder below is precisely the characteristic-$(T)$ realization of
Algorithm~\ref{alg:general-filter-decoder}. Throughout, bold uppercase
letters denote matrix tuples, bold lowercase letters the corresponding tuples
of skew residues, and plain uppercase letters their CRT lifts. We retain the
symbols $v$, $N$, and $\epsilon$ from the general decoder; here $N$ denotes
the filtered numerator, whereas $\mathsf N$ denotes the total skew degree. In
particular, if $E$ denotes the CRT lift associated with the error tuple
$\boldsymbol E$, then
$$
Y\equiv u+E\pmod{\Pi}.
$$
The error filter $v$ is chosen so that $vE$ is a multiple of $\Pi$. It then
satisfies
$$
vY\equiv vu\pmod{\Pi}.
$$
As in the general decoder, the key equation replaces the unknown product
$vu$ by a second unknown $N$. The degree constraints below guarantee that no
spurious nonzero solution can occur within the decoding radius.

\begin{theorem}
\label{thm:drinfeld-skew-crt-decoder}
Let $\epsilon\in\mathbb Z_{\geq 0}$ satisfy
$2\epsilon+t\leq\ell r$. Suppose that
$\boldsymbol Y=\boldsymbol C+\boldsymbol E$ is received, where
$\boldsymbol C$ is encoded by $u\in R$ with
$\deg_\tau(u)<K=mt$ and
$\wt_{SR}(\boldsymbol E)\leq\epsilon$. Let $Y$ be the CRT lift of
$\Theta^{-1}(\boldsymbol Y)$, with $\deg_\tau(Y)<\mathsf N$. The homogeneous
key equation
$$
vY\equiv N\pmod{\Pi},
\qquad
\deg_\tau(v)\leq m\epsilon,
\qquad
\deg_\tau(N)\leq m\epsilon+K-1,
$$
has a nonzero solution, and every nonzero solution satisfies $N=vu$. In
particular, $v\neq0$ and left division of $N$ by $v$ recovers $u$.
\end{theorem}

\begin{proof}
Let $E$ be the CRT lift of $\Theta^{-1}(\boldsymbol E)$. By
Theorem~\ref{thm:poly-skew-sum-rank-isometry},
$$
\wt_{\mathrm{ps}}\bigl(\Theta^{-1}(\boldsymbol E)\bigr)
\leq m\epsilon.
$$
If $E=0$, take $(v,N)=(1,u)$. Otherwise choose $v_0$ so that
$v_0E=\operatorname{lclm}(E,\Pi)$. The degree formula gives
$$
\deg_\tau(v_0)
=
\wt_{\mathrm{ps}}\bigl(\Theta^{-1}(\boldsymbol E)\bigr)
\leq m\epsilon,
$$
and $(v_0,v_0u)$ is a nonzero solution satisfying the degree constraints.

For any nonzero solution put $Z=N-vu$. Since
$Y\equiv u+E\pmod{\Pi}$ and $vY\equiv N\pmod{\Pi}$, there is $W\in R$
such that
$$
Z=vE+W\Pi.
$$
Any common right divisor of $\Pi$ and $E$ then also divides $Z$ on the
right. Hence
$$
\wt_{\mathrm{ps}}\bigl(\operatorname{CRT}_\Gamma(Z)\bigr)
\leq
\wt_{\mathrm{ps}}\bigl(\operatorname{CRT}_\Gamma(E)\bigr)
\leq m\epsilon.
$$
If $Z\neq0$, choose $U$ with
$UZ=\operatorname{lclm}(Z,\Pi)$. The degree formula and the preceding weight
bound give
$$
\deg_\tau(U)
=
\mathsf N-\deg_\tau\operatorname{gcrd}(\Pi,Z)
=
\wt_{\mathrm{ps}}\bigl(\operatorname{CRT}_\Gamma(Z)\bigr)
\leq m\epsilon.
$$
Furthermore, the degree constraints and $\deg_\tau(u)<K$ give
$$
\deg_\tau(Z)
\leq
m\epsilon+K-1.
$$
Therefore
$$
\deg_\tau(UZ)
\leq
2m\epsilon+K-1
\leq
\mathsf N-1.
$$
But $UZ$ is a nonzero left multiple of $\Pi$ in the domain $R$, so its
degree is at least $\mathsf N$, a contradiction. Thus $Z=0$ and $N=vu$.
If $v=0$, then $N\equiv0\pmod{\Pi}$. Moreover,
$$
\deg_\tau(N)
\leq m\epsilon+K-1
=
m(\epsilon+t)-1
\leq
\mathsf N-1.
$$
Hence $N=0$, contradicting nonzeroness.
\end{proof}

The preceding theorem gives the following bounded-distance decoder.

\begin{algorithm}
\label{alg:explicit-skew-crt-decoder}
Fix $\epsilon\in\mathbb Z_{\geq 0}$ with
$2\epsilon+t\leq\ell r$.

\emph{Input:} a received word
$\boldsymbol Y\in\Mat_{r\times r}(\vF_{q^m})^\ell$.
\emph{Output:} a skew-polynomial message $\widehat u$ of degree less than
$K$ whose codeword is at sum-rank distance at most $\epsilon$ from
$\boldsymbol Y$, or failure.

\begin{enumerate}
\item \emph{CRT-lifting step.} Compute
$\boldsymbol y=\Theta^{-1}(\boldsymbol Y)$ and its unique CRT lift $Y$ of
degree less than $\mathsf N$.

\item \emph{Filter step.} Find a nonzero solution $(v,N)$ of
$$
vY\equiv N\pmod{\Pi},
\qquad
\deg_\tau(v)\leq m\epsilon,
\qquad
\deg_\tau(N)\leq m\epsilon+K-1.
$$
Declare failure if none exists or if $v=0$.

\item \emph{Message-recovery step.} Divide $N$ with $v$ on the left,
obtaining
$$
N=v\widehat u+\Delta,
\qquad
\deg_\tau(\Delta)<\deg_\tau(v).
$$
Declare failure if $\Delta\neq0$ or
$\deg_\tau(\widehat u)\geq K$.

\item \emph{Verification step.} Declare failure if
$$
\wt_{SR}\left(
\boldsymbol Y-
\Theta\bigl(\operatorname{CRT}_\Gamma(\widehat u)\bigr)
\right)>\epsilon;
$$
otherwise return $\widehat u$.
\end{enumerate}
\end{algorithm}

\begin{corollary}
Algorithm~\ref{alg:explicit-skew-crt-decoder} corrects every sum-rank error
of weight at most
$$
\left\lfloor\frac{\ell r-t}{2}\right\rfloor.
$$
After the CRT data and the maps $\Theta_i$ have been precomputed, its
key-equation step costs $O(\mathsf N^\omega)$ operations over
$\vF_{q^r}$, where $\omega$ denotes the exponent of matrix multiplication.
\end{corollary}

\begin{proof}
The radius follows from Theorem~\ref{thm:drinfeld-skew-crt-decoder}. The key
equation has $\mathsf N$ scalar equations and $2m\epsilon+K+1$ unknowns over
$\vF_{q^r}$. Since $2m\epsilon+K\leq\mathsf N$, it has at most
$\mathsf N+1$ columns, and standard linear algebra gives the stated
complexity, consistently with \cite[Proposition~8]{NR26}. The displayed
$O(\mathsf N^\omega)$ bound concerns the online dense linear-algebra step.
The construction of the torsion bases, the maps $\Theta_i$, and the CRT data
is precomputation; with effective finite-field representations fixed, the
coordinate conversions and Ore-polynomial divisions are polynomial-time
operations.
\end{proof}

\section{Cryptographic Outlook}

\begin{remark}
The Drinfeld-module construction provides a natural separation between a
public code and its efficient decoding description.  One may publish only an
$\vF_q$-generator matrix of a sum-rank-isometric image of the code, while
retaining the Drinfeld module, the torsion primes and bases, and the
restriction or CRT data as a private decoding key.  This is precisely the type
of hidden algebraic description sought in McEliece-type cryptosystems.

Classical proposals based directly on Reed-Solomon, Gabidulin, and
linearized Reed-Solomon codes face structural attacks exploiting,
respectively, Schur products, Frobenius closures, and their sum-rank analogues;
in several regimes these methods recover enough defining data to reconstruct
an efficient decoder \cite{BartzEtAl22,HBH23}.  In the present construction,
the private arithmetic data are not published as evaluation locators or as a
Moore matrix, so these recovery procedures do not immediately apply in the
same form. Especially for general supersingular Drinfeld modules, the
choice of the module and of several torsion levels provides a rich family of
possible hidden descriptions.  Moreover, recognizing that the public code is
structured would not by itself reveal the private Drinfeld/CRT realization or
the corresponding decoder.

These observations are heuristic and do not constitute a security proof.

This leads to a natural Drinfeld/CRT reconstruction problem: starting only
from a public generator matrix, recover an equivalent Drinfeld module,
torsion and restriction data, or any other efficient decoder.  If this
reconstruction problem is hard, the hidden arithmetic realization offers a
genuine potential advantage over direct Gabidulin- or linearized
Reed-Solomon-based choices.  The MSRD property and the decoder developed
above further make these codes natural candidates for secure multishot
constructions analogous to \cite{MPK19}, once suitable nested or dual
families are selected.  Thus Drinfeld-module codes provide a promising
framework for both public-key and network-coding-oriented cryptographic
constructions.
\end{remark}

\section*{Acknowledgments}

\noindent G. Micheli was supported by NSF CAREER grant 2338424.

\noindent M. Papikian was supported in part by the Simons Foundation,
award number MPS-TSM-00008093.

\medskip

\noindent\emph{Use of generative AI.}
During the preparation of this manuscript, OpenAI's ChatGPT was used as an
editorial aid for linguistic and expository revision, notational consistency,
and LaTeX formatting. All mathematical ideas, results, and proofs are entirely
the authors' own work.

\bibliographystyle{plain}
\bibliography{references}
\end{document}